\documentclass[a4paper,12pt,twoside]{amsart}
\usepackage{amsmath,amsthm,amsfonts,amssymb, mathrsfs, amsxtra}
\usepackage{verbatim,color}
\usepackage[T1]{fontenc}
\usepackage[norelsize]{algorithm2e}

\newcommand{\CR}{{}}

\allowdisplaybreaks

\usepackage{ulem}

\let\oldmarginpar\marginpar
\renewcommand\marginpar[1]{\-\oldmarginpar[\raggedleft\footnotesize #1]%
{\raggedright\footnotesize #1}}
\newtheorem{theorem}{Theorem}

\newtheorem{corollary}{Corollary}
\theoremstyle{definition}

\newtheorem{example}{Example}
\newtheorem{question}{Question}

\theoremstyle{remark}
\newtheorem{remark}{Remark}

\newcommand{\Z}{\mathbb{Z}}
\newcommand{\D}{\mathbb{D}}
\newcommand{\abs}[1]{| #1 |}

\newcommand{\Bigabs}[1]{\Big| #1 \Big|}

\newcommand{\N}{\mathbb{N}}
\newcommand{\R}{\mathbb{R}}
\newcommand{\T}{\mathbb{T}}
\newcommand{\C}{\mathbb{C}}

\newcommand{\Hp}{\mathscr{H}}

\def\T{\mathbb{T}}
\def\N{\mathbb{N}}
\def\Z{\mathbb{Z}}

\def\R{\mathbb{R}}
\def\C{\mathbb{C}}

\def\1{\mathbf{1}}

\newcommand{\dif}{\mathrm{d}}
\newcommand{\e}{\mathrm{e}}
\newcommand{\im}{\mathrm{i}}
\newcommand{\norm}[1]{\|#1\|}

\renewcommand{\Re}{\operatorname{Re}}

\newcommand{\beqno}{\begin{eqnarray*}}
\newcommand{\eeqno}{\end{eqnarray*}}
\newcommand{\beqla}[1] {\begin {eqnarray}\label{#1}}
\def\eeq {\end {eqnarray}}
\newcommand{\beq}{\begin {eqnarray}}

\newcommand{\torus}{{\mathbb T}}

\begin{document}

 \title[Fatou and  brothers Riesz theorems in $\T^\infty$]{Fatou and  brothers Riesz theorems in \CR the infinite-dimensional polydisc}
 
\author[Aleman]{Alexandru Aleman} 
\address{Centre for Mathematical Sciences, Lund University, P.O. Box 118, SE-221 00 Lund, Sweden}
\email{aleman@maths.lth.se}

\author[Olsen]{Jan-Fredrik Olsen} 
\address{Centre for Mathematical Sciences, Lund University, P.O. Box 118, SE-221 00 Lund, Sweden}
\email{janfreol@maths.lth.se}

\author[Saksman]{Eero Saksman}
\address{Department of Mathematics and Statistics, University of Helsinki, PO~Box~68, FI-00014 Helsinki, Finland}
\email{eero.saksman@helsinki.fi}

\thanks{The third author was
supported \CR by the Finnish Academy CoE in Analysis and Dynamics Research and by a Knut and Alice Wallenberg Grant.}


\begin{abstract} We study the boundary behavior of functions in the Hardy spaces on the infinite dimensional polydisk. These spaces are intimately related to the Hardy spaces of Dirichlet series.  We exhibit several Fatou  and Marcinkiewicz-Zygmund type theorems for radial convergence. As a consequence one obtains easy new proofs of {\CR the brothers F. and M. Riesz Theorems in infinite dimensions}, \textcolor{black}{as well as being able to extend a result of Rudin concerning which functions are equal to the modulus of an $H^1$ function almost everywhere to $\T^\infty$.} Finally, we provide counterexamples showing that the pointwise  Fatou theorem is not true in infinite dimensions without restrictions to the mode of radial convergence even for bounded  analytic functions.
\end{abstract}

\maketitle

\section{Introduction}

The object of study in this paper is the Hardy spaces $H^p$ on the infinite dimensional torus $\T^\infty = \{ (z_1, z_2, \ldots) : z_n \in \T\}$.  In recent years, there has been a renewed interest in these spaces, mainly due to their connection to Dirichlet series  and thereby to analytic number theory. We refer to \cite{queffelec2013} for  the related  theory of Dirichlet series and for basic references to the field.

In order to recall the definition of the space $H^p(\T^\infty)$ for $p\in [1,\infty],$ observe that $\T^\infty$ is a compact abelian group  with dual $\Z^\infty$ and Haar measure  $\dif \theta = \dif \theta_1 \times \dif \theta_2 \times \cdots$,  where $\dif \theta_n$ is the \textit{normalised} Haar measure on the $n$-th copy of $\T$. Elements  $f$ in $L^p(\T^\infty)$ are uniquely defined by their Fourier series expansion (see, e.g., \cite{rudin1962})
\begin{equation*}
	  f \sim \sum_{\nu \in \Z^\infty_0} a_\nu \e^{\im \theta \cdot \nu},
\end{equation*}
where the Fourier coefficients are defined in the standard manner {\CR  and $\nu\in\Z^\infty_0$ means that only finitely many of the components of  the index sequence $\nu$ are non-zero}.
One may now  define the Hardy spaces $H^p$ to be the analytic part of $L^p$ in the following way
\begin{equation*}
	H^p(\T^\infty) = \Big\{  f \in L^p(\T^\infty) : f   \sim \sum_{\nu \in \N^\infty_0} a_\nu \e^{\im \theta \cdot \nu}  \Big\}.
\end{equation*}
Note that also other notions of analyticity are possible in this setting (see, e.g., our Corollary \ref{corr-big}).

 A basic, and extremely useful, feature of the one variable theory is that  any function  $f\in H^p(\torus)$  can be extended to an analytic function on the open unit disc $\D$.   In particular,   the function $\e^{\im t} \longmapsto f(r\e^{\im t})$ is smooth and approximates the function $f$ in norm as  $r\nearrow 1$,  (weak-$\ast$, if $p=\infty$) and, for almost every $\e^{\im t}\in\T$, it holds that $f(re^{\im t})\to f(e^{\im t})$. This remains true in finite dimensions with almost no restrictions to the radial (or even non-tangential) approach, see Remark \ref{re:rudin} and Corollary \ref{re:rudin} below. 
 
 The purpose of the current paper is to initiate the investigation  of to which extent such Fatou-type approximations   hold in the infinite dimensional setting. We note that \cite{saksman_seip2009} contains some first steps in this direction for the space $H^\infty(\T^\infty).$
 
 While $\T^\infty$ is the distinguished boundary of $\D^\infty = (z_1, z_2, \ldots) : z_n \in \D\}$, it is no longer straight-forward to extend functions $f \in H^p(\T^\infty)$ to functions on the polydisc $\D^\infty$. 
 This is because   point evaluations for Hardy functions in the polydisk are well-defined only in $\ell^2\cap \D^\infty$ for $p<\infty$, see 
\cite{cole_gamelin1986}, and in $c_0\cap \D^\infty$ for $p=\infty$, see \cite{hilbert1909}. In particular, when formulating Fatou-type results, these restrictions have to be kept in mind.


Our first result, Theorem \ref{th:fatou} below, considers a boundary approach of the type
$$
(re^{i\theta_1},r^2e^{i\theta_2},r^3e^{i\theta_3},\ldots) \qquad \textrm{with}\quad r\nearrow 1^-,
$$
and shows that the standard Fatou type results remain  valid  for functions with Fourier spectrum supported on $\N^\infty_0 \cup (-\N^\infty_0)$. As a corollary, one obtains easy proofs of   infinite dimensional versions of some results due to the brothers M. and F. Riesz, see Corollaries \ref{co:riesz1} and \ref{co:riesz2}. Corollary \ref{co:hardy} yields a useful characterisation of elements $f\in H^1(\T^\infty)$ in terms of uniform $L^1$-boundedness of their the '$m$te Abschnitt'. Theorem \ref{th:infiniterudin} generalizes  to infinite dimensions the theorem of Marcinkiewicz and Zygmund concerning the  vanishing radial limits of singular measures  that have Fourier series supported on $\N^\infty_0 \cup (-\N^\infty_0)$. Finally,    {\CR Section \ref {se:examples} }  provides counter examples to unrestricted radial approach in infinite dimensions and poses some open questions. In particular, Theorem \ref{th:example}   yields a bounded analytic function with no boundary limit at almost every boundary point for a suitable radial approach.

\section{Fatou and  brother Riesz theorems in $\T^\infty$} \label{riesz section}

\newcommand{\Wi}{{\rm Wi}}
\newcommand{\TV}{{\rm TV}}

The inspiration for the present  paper, as well as parts of the above-cited paper \cite{saksman_seip2009}, is the work of  Helson \cite{helson1969}. He  introduced so-called vertical limit functions to   the theory of Dirichlet series in order to extend them analytically up to the imaginary axis. Somewhat simplified, he showed that a Dirichlet series
\begin{equation*}
	F(s) = \sum_{n=1}^\infty a_n n^{-s},
\end{equation*}
for which   $(a_n)_{n=1}^\infty \in \ell^2$ can,  in a   weak sense, be extended to all of $\Re s > 0$. This is not immediately clear, since such   Dirichlet series, in general, will only converge when $\Re s > 1/2$ ({\CR this} follows, e.g., from Cauchy-Schwarz).  However, the effect of  taking vertical limits of $F(s)$  is to replace the coefficients $(a_n)_{n=1}^\infty$ by $(\chi(n) a_n)_{n=1}^\infty$, where $n \longmapsto \chi(n)$ is a function from $\N$ to $\T$ that is multiplicative in the sense that $\chi(nm) = \chi(n)\chi(m)$ for all $n, m \in \Z$.  The statement of Helson is essentially that for almost every choice of $\chi$, the modified Dirichlet series, which we may denote by $F_\chi$, has an analytic extension up to the imaginary axis. {\CR Helson's proof  combines Fubini with one variable results to analytically extend  $F_\chi(\im t)$ to the right-half plane.}

{\CR The relation of Helson's result  to Hardy spaces $H^p(\T^\infty)$  takes place through the fundamental connection between  Dirichlet series and Hardy spaces on the polydisc,  due to  H.\,Bohr  \cite{bohr1913lift}.}  Indeed, functions on $\T^\infty$ formally become Dirichlet series when restricted to the path $t \mapsto (p_n^{-\im t})$, where $p_n$ is the $n$-th prime number. Explicitly, 
 \begin{equation*}
	 f \sim \sum_{\nu \in {\CR \N^\infty_0}} a_\nu \e^{\im \theta \cdot \nu} \implies F(\im t) := f(p_1^{\im t}, p_2^{\im t}, \ldots)  \sim    \sum_{n=1}^\infty a_{n} n^{-\im t},
\end{equation*}
where $n = p_1^{\nu_1} \cdots p_k^{\nu_k}$ and we identify $a_\nu$ to the corresponding coefficient $a_n$. Also note that for $\sigma>1/2$, the slightly modified path $t \mapsto (p_n^{-\im t  - \sigma})$ lies in $\D^\infty  \cap \ell^2$, and so, by the above mentioned results on bounded point evaluations, every $f \in H^p(\T^\infty)$ restricts to an analytic Dirichlet series. These   restrictions exactly  form the Dirichlet-Hardy spaces $\Hp^p$. {\CR Helson's result can then be reformulated as follows:}
 For almost every $\chi \in \T^\infty$, the restriction of a function $f \in H^{\CR p}(\T^\infty)$ to the path  $t \mapsto (\chi(p_n) p_n^{-\im t -\sigma})$ gives an analytic function on $\Re(\sigma + \im t) = \sigma  >0$.  

 {\CR A similar scheme  was used} in    \cite{saksman_seip2009} to approximate functions   $f \in H^\infty(\T^\infty)$ almost everywhere. First, note that   the restriction $F(s)$ is analytic on $\C_+ = \{ \Re s >  0 \}$, due to \cite{hilbert1909} and \cite{bohr1913lift} (also, in this connection see \cite{hls1997}).  Next,    fix $\chi = \e^{\im \theta_0}$, and extend  $\theta\mapsto f_{\e^{\im \theta_0}}(\e^{\im \theta}) := f(\e^{\im (\theta+\theta_0)})$ to the analytic function  $F_{\e^{\im \theta_0}}(s)$. Reversing the roles of the variables, put $\tilde{f}_{s}(\e^{\im \theta_0}) := F_{\e^{\im \theta_0}}(s)$.   {\CR Applying  ergodicity, one may now show that $\tilde{f}_s$ tends to $f$ almost everywhere as $s \rightarrow 0$ non-tangentially. Hence results from one-dimensional theory    can be used  to   deduce results on $H^\infty(\T^\infty)$. }
However, while {\CR in principle}  still feasible for $p \in [1,\infty)$, this approach  becomes cumbersome since $F_{\e^{\im \theta}}(s)$ is only defined on the strip $0 < \Re s \leq 1/2$   for almost every $\e^{\im \theta_0}$, and so the resulting   Fatou-type statements   are far from trivial.


{\CR One of the aims  of this note} is  to describe an alternative  approach to Fatou-type approximation\footnote{We are especially interested in the pointwise convergence at a boundary point since,  by using the density of polynomials, \CR practically any   natural and well-defined  radial approximation scheme leads to approximation in the $L^p$ norm for $p\in (1,\infty)$, and in the ${\rm weak}^*$ sense for $p=\infty.$} of    certain functions on the polydisc, which is in the same spirit as the extensions mentioned above{\CR , but is easier to deal with and yields stronger results}.  {\CR Our} idea is simple: given the function $f(\e^{\im \theta_1},\e^{\im \theta_2},\ldots)$, define a family of functions
\begin{equation}\label{eq:app}
f_\xi(\e^{\im \theta}):=f(\xi\e^{\im \theta_1},\xi^2\e^{\im \theta_2},\ldots),\quad \e^{\im \theta}\in \T^\infty,
\end{equation}
where $\xi\in \D$ is a complex parameter from the unit disc. We will soon define $f_\xi$ in a precise manner, but one should note that $f_\xi(\e^{\im \theta})$ is also well-defined pointwise by the mere fact that 
$(\xi\e^{\im \theta_1},\xi^2 \e^{\im \theta_2},\ldots)\in\ell^2$ since $|\xi|<1.$  The usefulness of introducing $f_\xi$ lies in the possibility of  fixing $\e^{\im \theta}\in \T^\infty$ and  employing, with a slight abuse of notation, the function of one variable $\xi\mapsto f_\xi(\e^{\im \theta})$ in order to transfer one-dimensional tools to the infinite-dimensional situation. Note that    $f_\xi(\e^{\im \theta})$ is harmonic with respect to $\xi \in \D$ exactly when $f$ has Fourier spectrum supported on $\N^\infty_0 \cup (-\N^\infty_0)$.

Let us begin by considering Poisson extensions of general measures. We define the space $\Wi(\T^\infty )\subset C(\T^\infty )$ consisting of the continuous functions $f:\torus^\infty\to\C$ with absolutely convergent Fourier series, i.e., with $\sum_{\nu \in \Z^\infty_0} |\widehat f(\nu)| <\infty.$ For any such $f$, the absolute convergence ensures that its unique polyharmonic extension to  $ \D^\infty$ is well-defined and actually continuous in  all of the closure $\overline{ \D}^\infty$. With yet a slight abuse of notation, we also denote this extension by $f$ so that for any $z=(\rho_1e^{i\theta_1},\rho_2e^{i\theta_2}\ldots )\in \overline{ { \D}}^\infty$, we have
\begin{equation}\label{eq:poly1}
f(z)=\sum_{\nu \in \Z^\infty_0}\widehat f(\nu)\rho^{|\nu|} e^{i\nu\cdot\theta}. 
\end{equation}
Here, we   employ the abbreviations  
 $$
 \rho^{|\nu|}:=\rho_1^{|\nu_1|}\rho_2^{|\nu_2|}\ldots \rho_\ell^{|\nu_\ell|} \quad {\rm and} \quad  |\nu|_1:=\sum_{j=1}^\ell \abs{\nu_j},
 $$
for any  multi-index $\nu=(\nu_1,\ldots,\nu_\ell)\in \Z^\infty_0$. 

Let $\mu$ be finite Borel measure on $\T^\infty$,  {and set $ \nu^*:=(\nu_1,2\nu_2,\ldots, \ell\nu_\ell)$.} Then for any $\xi=r e^{it}\in\overline{\D}$, we set, in accordance with (\ref{eq:app}),
\begin{equation}\label{eq:mu}
\mu_\xi=\sum_{\nu \in \Z^\infty_0}\widehat\mu(\nu)r^{|\nu^*|_1}e^{it(\nu_1+2\nu_2+\ldots )}e^{i\nu\cdot\theta}.
\end{equation}
Then, it holds that
$\mu_\xi\in \Wi(\T^\infty)$   for  all $\xi\in\D$, as 
one may compute
\begin{equation}\label{eq:mu3}
\|\mu_\xi\|_\Wi\leq c\sum_{\nu \in \Z^\infty_0}|\xi|^{|\nu^*|_1}=c\prod_{j=1}^\infty
\left(1+2\sum_{k=1}^\infty |\xi|^{kj}\right)=c\prod_{j=1}^\infty \left(\frac{1+|\xi|^j}{1-|\xi|^j}\right) <\infty.
\end{equation}
We may then define a `radial' maximal function of $\mu$ at every point $e^{i\theta}\in\T^\infty$ via
$$
M\mu(e^{i\theta}):=\sup_{r\in(0,1)}|\mu_r(e^{i\theta})|. 
 $$
\begin{remark}\label{remark1}
	Observe that although it follows from the above argument that the function $\xi \mapsto \mu_\xi$ is continuous  for every fixed $\e^{\im \theta} \in \T^\infty$, it is harmonic with respect to the variable $\xi \in \D$ if and only if the Fourier transform of $\mu$ has     support on the set $\N^\infty_0 \cup (-\N^\infty_0)$.
\end{remark}

e denote by   $A_m\mu$ Bohr's   '$m$te Abschnitt' of the measure $\mu$, which is defined by  $$A_m \mu \sim \sum_{\eta \in \Z^m}  \widehat{\mu}(\widetilde{\eta}) \e^{\im \theta \cdot \widetilde \eta},$$ where $\widetilde\eta=(\eta_1,\ldots \eta_m,0,0,\ldots)$ for $\eta \in \Z^m.$ In other words, the harmonic extensions satisfy
$A_m\mu(z)=\mu (z_1,\ldots, z_m,0,0,\ldots)$ for any $z\in\ell^1\cap \D^\infty.$

\begin{theorem}\label{th:fatou} \hspace{1 cm}

\noindent{\bf (i)}\quad For any  finite Borel measure $\mu$ on $\T^\infty $, one has
\begin{equation} 
 \int_{\torus^\infty}|\mu_\xi(\e^{\im \theta})| \dif \theta\leq \|\mu\|_{TV}\; \; {\rm for\; all}\;\; \xi\in\D \quad {\and}\quad \mu_r\stackrel{w^*}{\longrightarrow}\mu\quad {\rm as}\; r\nearrow 1.
 \end{equation}

\noindent{\bf (ii)}\quad When restricted to measures with Fourier spectrum supported   on $\N^\infty_0 \cup (-\N^\infty_0)$, $M$ is of weak type 1-1, i.e., there is $C<\infty$ such that for any  finite Borel measure $\mu$ on $\T^\infty $,  with Fourier spectrum supported on $\N^\infty_0 \cup (-\N^\infty_0)$, and $\lambda >0$, it holds that
\begin{equation}\label{eq:weak1}
\Bigabs{\big\{ \e^{\im \theta} \in \T^\infty : M\mu(\e^{\im \theta})>\lambda\}}\leq \frac{C\|\mu\|_\TV }{\lambda}.
\end{equation}
Moreover, the finite radial limit $\mu^*(\e^{\im \theta}):=\lim_{r\to 1^-}
\mu_r(\e^{\im \theta})$ exists for almost every $\e^{\im \theta}\in\T^\infty .$

\noindent{\bf (iii)}\quad For any  $f\in L^1(\T^\infty)$  with Fourier spectrum supported on $\N^\infty_0 \cup (-\N^\infty_0)$, one has 
$f(\e^{\im \theta})=\lim_{r\to 1}f_r(\e^{\im \theta})$  for  a.e. $\e^{\im \theta}\in\T^\infty.$
Moreover, $\|f_r-f\|_1\to 0$ as $r\nearrow 1.$

\noindent{\bf (iv)}\quad There is $C<\infty$ such that for any $f\in H^1(\T^\infty)$,
\begin{equation}\label{eq:weak3}
\|Mf\|_1\leq C\| f\|_{H^1(\T^\infty)}.
\end{equation}
\textcolor{black}{Similarly, for  analytic measures $\mu$  on $\T^\infty$, that is, measures $\mu$ with
$\widehat\mu (\nu)=0$ if $\nu\not\in \N^\infty_0 $, there exists a constant $C<\infty$ so that
\begin{equation}
\|M \mu \|_1\leq C\| \mu \|_{TV}.
\end{equation}}

\end{theorem}

\begin{proof}
(i)\quad Because $\mu_\xi\in \Wi(\T^\infty )$, it is enough to verify that the   '$m$te Abschnitt'
$A_m\mu_\xi$ satisfies $\| A_m\mu_\xi\|_1\leq \|\mu\|_{TV}$ for every integer $m\geq 1.$ Indeed, this follows immediately by observing that
$$
A_m\mu_\xi= \big(P_1(\cdot,\xi)\ldots P_m(\cdot, \xi^m)\big)*\mu,
$$
where $P_j(\cdot,\xi)$ denotes the Poisson kernel on $\T$ at $\xi$ with respect to the $j$'th variable. 
The second statement follows from this bound and the very definition of $\mu_r$.

(ii)\quad  For each fixed $\e^{\im \theta}\in \T^\infty$, consider the function $g_\theta$ on $\D$, where for any $\xi\in\D$ we set $g_\theta (\xi):=\mu_\xi(\e^{\im \theta}).$ Formula \eqref{eq:mu} and the estimate  \eqref{eq:mu3}  verify that the Fourier development of $g_\theta$ converges uniformly in compact subsets of $\{ |\xi| <1\}$ and hence,  by Remark \ref{remark1},
$g_\theta$ is harmonic in $\D.$ Observe that the map $(\e^{\im \theta_1}, \e^{\im \theta_2},\ldots)\mapsto 
(e^{i(\theta_1+t)},e^{i(\theta_2 + 2t)},\ldots )$ is a measure-preserving homeomorphism. Hence, by Fubini,  we may compute, for any $r\in (0,1)$,
$$
\int_{\T^\infty}\left(\int_\T|g_\theta(re^{it})|\frac{dt}{2\pi}\right)\dif \theta
 =\int_\T\|\mu_{re^{it}}
\|_1 \frac{\dif t}{2\pi} \leq \|\mu\|_{\TV}<\infty .
$$
Letting $r\nearrow 1$, we obtain by Fatou's lemma
\begin{equation}\label{eq:fatou5}
\int_{\T^\infty}\|g_\theta\|_{h^1(\T)} \dif \theta=\int_{\T^\infty}\|g_\theta\|_{\TV} \dif \theta \leq \|\mu\|_{TV}.
\end{equation}
Thus, for a.e. $\e^{\im \theta}\in\T^\infty$, the function $g_\theta$ is the Poisson-extension of a finite measure (also denoted above by $g_\theta$) to $\D.$ Especially, for all these ${\theta}$ we deduce the  existence  of the finite limit $\lim_{r\to 1^-} g_\theta(r\e^{\im t})$ for almost every $\e^{\im t }\in \T$. By Fubini, there is at least one fixed $\e^{\im t_0}\in\T$ so that $\lim_{r\to 1^-}g_\theta(r\e^{\im t_0})=
\lim_{r\to 1^-}\mu_{r\e^{\im t_0}}(\e^{\im \theta})$ exists   a.e. $\e^{\im \theta} \in \T^\infty$, whence
the finite limit $
\lim_{r\to 1^-}\mu_{r}(\e^{\im \theta})$ exists almost everywhere.

Denote by $M_1$ the radial maximal function in the unit disc, and recall the one-dimensional strong to weak maximal inequality $\abs{\{\e^{\im t} \in \T : M_1 \eta(\e^{\im t})>\lambda\}}\leq C\lambda^{-1}\|\eta\|_{TV}$ for $\lambda >0.$ We then obtain by Fubini and \eqref{eq:fatou5}
\begin{equation*}
\begin{split}
&\Bigabs{\big\{\e^{\im \theta} \in\T^\infty\;:\;  M\mu(\e^{\im \theta})>\lambda\big\}}
=\int_{\T^\infty}\chi_{\{\e^{\im \theta}\;:\;  M\mu(\e^{\im \theta})>\lambda\}} \dif \theta\\
&=\int_\T\left(\int_{\T^\infty}\chi_{\{\e^{\im \theta}\;:\;  M\mu_{\e^{\im t}}(\e^{\im \theta})>\lambda\}} \dif \theta \right) \frac{\dif t }{2\pi}
=\int_{\T^\infty}\left(\int_\T\chi_{\{\e^{\im \theta}\;:\;  M_1\mu_{\e^{\im \theta}}(\e^{\im t})>\lambda\} }\frac{\dif t }{2\pi}\right) \dif \theta \\
&=\int_{\T^\infty}\abs{\big\{ \e^{\im t}\in \T\;:\;  M_1g_\theta(\e^{\im t})>\lambda\big\}} \dif \theta
\leq C\lambda^{-1}\int_{\T^\infty}\|g_\theta\|_{\TV} \dif \theta
\leq C\lambda^{-1}\|\mu\|_\TV .
\end{split}
\end{equation*}

(iii)\quad The claim follows in a standard manner from the weak-type inequality in part (ii), and the fact that finite trigonometric polynomials are dense in $L^1(\T^\infty )$, see e.g. \cite[Theorem I.5.3]{garnett1981}.

(iv)\quad Assume that $\mu$ is an analytic Borel measure on $\T^\infty$.  This time the functions $g_\theta$ defined in the beginning of the proof of part (ii) are analytic in $\D$ and the counterpart of \eqref{eq:fatou5} reads
\begin{equation}\label{eq:fatou}
\int_{\T^\infty}\|g_\theta\|_{H^1(\T)}\leq C\|\mu\|_{TV}.
\end{equation}
Thus, for almost every $\e^{\im \theta}\in \T^\infty$, we have $g_\theta\in H^1(\T).$
Consequently, for almost every $\e^{\im \theta}\in\T^\infty$, the finite limit $\lim_{r\to 1^{-}}g_{r\e^{\im t}}(\e^{\im \theta})$ exists for almost every $\e^{\im t}\in\T .$ Then, by Fubini,  the limit
$\lim_{r\to 1^{-}}g_{r}(\e^{\im \theta})$ exits for almost every $\e^{\im \theta}\in\T^\infty$ (of course this follows also from part (ii) of the theorem).  We call this function $g$.

Let $C$ stand for the finite constant in the 1-dimensional Fefferman-Stein  radial maximal inequality
$\| M_1h\|_{L^1(\T)}\leq C\|h\|_{H^1(\T)}.$ Then, Fubini and \eqref{eq:fatou} yield immediately the desired inequality
\begin{equation*}
\begin{split}
&\int_{\T^\infty}Mg(\e^{\im \theta}) \dif \theta=\int_{\T}\left(\int_{\T^\infty}M_1g_\theta(\e^{\im t}) \dif \theta\right) \frac{\dif t}{2 \pi} =\\ &\int_{\T^\infty}\left(\int_{\T}M_1g_\theta(\e^{\im t}) \frac{\dif t}{2\pi} \right) \dif \theta
\leq C\int_{\T^\infty}\|g_\theta\|_{H^1(\T)} \dif \theta \leq C\|\mu\|_\TV .
\end{split}
\end{equation*}
\end{proof}

As a corollary, we obtain new proofs of   two ``brothers Riesz'' theorems in infinite dimensions. Their generalization to $H^p$ spaces on groups was obtained in a very involved paper by Helson and Lowdenslager \cite{helson_lowdenslager1958}. Observe that our proof of the first result uses only parts (i) and (iv) of the previous theorem (and their proofs are independent of parts (ii) and (iii))
\begin{corollary}[F. and M. Riesz theorem I]\label{co:riesz1}
Every analytic measure $\mu$ on $\T^\infty$ is absolutely continuous.
\end{corollary}
\begin{proof}
In part (iv) of Theorem \ref{th:fatou}, we see that $\mu_r\to g$ a.e. pointwise on $\T^\infty,$ for some $g \in H^1(\T^\infty)$, and we obtain that $\|\mu_r-g\|_1\to 0$ as $r\nearrow 1$ by employing the integrable majorant $M\mu.$ Hence $\mu$ is absolutely continuous with density $g.$ 
\end{proof}

\begin{corollary} \label{co:riesz2}
If $f \in H^1(\T^\infty)$, then $\log \abs{f} \in L^1(\T^\infty)$.
\end{corollary}
\begin{proof}
First, assume that $f(0) \neq 0$. For any $g\in H^1(\T)$ with $g(0)\not=0$, it is classical that
\begin{equation}\label{eq:sh}
-\|g\|_{H^1}\leq \int_\T \log\left(\frac{1}{|g(\e^{\im t})|}\right) \frac{\dif t}{2\pi} \leq \log\left(\frac{1}{|g(0)|}\right).
\end{equation}
Actually, this follows directly from
the  superharmonicity of $\log(1/|g|)$ and the simple inequality $-x\leq \log(1/x)$ for $x>0$. Let again $f_\theta(\xi)=f_\xi(\e^{\im \theta})$ for $\xi\in\D$ and $\e^{\im \theta }\in \T^\infty$, and observe that part (iii) of Theorem \ref{th:fatou} verifies that the (a.s. in $\e^{\im t}$) boundary values $f_\theta(\e^{\im  t})$ satisfy $f_\theta(\e^{\im t}) =f(\e^{\im (\theta_1 + t)}, \e^{\im (\theta_2 + 2t)}, \ldots)$ for almost every $(\e^{\im t},\e^{\im \theta})\in \T\times \T^\infty.$
The claim of the theorem follows simply by Fubini after one substitutes $g=f_\theta$
in \eqref{eq:sh}, integrates over $\T^\infty$ and observes that $f_\theta(0)=  f(0)\not=0$ for all $\e^{\im \theta}\in\T^\infty .$ Here, the left-hand inequality is just used to assure integrability.

Suppose $f(0)= 0$. If $f$ is not identically equal to zero, the same holds true for some abschnitt of $f$, i.e., there exist $z_1, \ldots, z_k \in \D$ such that $f(z_1, \ldots, z_d, 0, 0, \ldots)$ is non-zero. Composing $f$ with appropriate M\"obius transforms $\phi_i$, each sending the origin to $z_i$ in the $k$ first variables,  we obtain the desired conclusion   by  applying the above argument to the resulting function $f(\phi_1(z_1), \ldots, \phi_k(z_k), z_{k+1}, \ldots)$.
\end{proof}

By a standard ${\rm weak}^*$-convergence argument we obtain the following useful statement as a consequence of Corollary \ref{co:riesz1}.

\textcolor{black}{
\begin{corollary} \label{co:hardy}
Assume that $\mu$ is a formal Fourier series with spectrum supported on $\N^\infty_0$   such that $A_m \mu \in H^1(\T^\infty)$ with $\|A_m \mu\|_{H^1(\T^m)}\leq C$ for all $m\geq 1$. Then $\mu=f$ for some $f\in H^1(\T^\infty)$ with
$\|f\|_{H^1(\T^\infty)}\leq C$.
\end{corollary}}
\noindent The same statement is naturally also true for $p>1$, but then the brothers Riesz theorem is not needed in the proof.

\begin{remark}\label{re:Helson}
The method used in Theorem \ref{th:fatou} is  considerably simpler and gives stronger results than the Helson type approach using the Dirichlet series mentioned in the beginning of this section. However, it lacks a couple of beautiful features of the latter one. For instance, in the Helson method the torus $\T^\infty$ is divided in orbits (corresponding to a parameter $t\in\R$) that are {\it ergodic} with respect to the basic measure $m$. \textcolor{black}{Moreover,  the approach of Helson uniquely maps functions on $\T^\infty$ to functions of one complex variable. Our method lacks this uniqueness property. Indeed, restricting functions on $\T^\infty$ to   $(\xi, \xi^2, \xi^3, \ldots)$ maps, say, the monomials $z_1^2$ and $z_2$ to the same one dimensional function $\xi \mapsto \xi^2$.}
\end{remark}

We next give an infinite-dimensional version of a theorem of Marcinkiewicz and Zygmund \cite[Theorem 2.3.1]{rudin1969} (see also \cite[Chapter 17]{zygmund1968}).
\begin{theorem}\label{th:infiniterudin} \textcolor{black}{Let $\mu$ be a Borel measure on $\T^\infty$ 
with decomposition $d\mu= fd  m+ \dif \mu_s$, where   both $f \in L^1(\T^\infty)$ and the singular part $\mu_s$ have Fourier spectrum supported on $\N^\infty_0 \cup (-\N^\infty_0)$.     Then,  for almost every $e^{it}\in \T^\infty$, it holds that $$
\mu^*(e^{\im t})=f(e^{\im t}).
$$}
\end{theorem}
\begin{proof}
In light of Theorem \ref{th:fatou}, parts (ii) and (iii), we may assume that $\mu$ is singular. In addition, we may  assume that $\mu$ is positive.  We define
 \newcommand{\PO}{\mathcal{P}}
$$
\textcolor{black}{\mu^*(e^{i\theta})=\lim_{r\to 1^-} [\PO_r \mu] (\e^{\im\theta}),}
$$
where $\PO_r\mu$ is the convolution of $\mu$ against 
 the product kernel
 \newcommand{\PK}{\mathcal{P}}
\begin{equation*}
	 \PO_r(\theta) = \prod_n \frac{1-r^{2n}}{1 - 2r^n \cos \theta_n + r^{2n}}.
\end{equation*}
Simple estimates verify that $\PK_r$ is continuous on $\T^\infty$ for all $r<1.$
Let  $f$ be any positive
continuous function on the infinite dimensional polydisk, and fix an arbitrary sequence $r_k\nearrow 1^-$ as $k\to\infty.$
We may argue by using  the dominated convergence theorem, Fubini's theorem, and finally  Fatou's lemma that
\textcolor{black}{ \begin{eqnarray*}
 \int f \dif \mu &=& \lim_{k\to\infty} \int [\PO_{r_k} f] \dif \mu\; =\; \lim_{k\to\infty}  \int f [\PO_{r_k} \mu] \,\dif \theta   \; \geq \; \int \mu^\ast f \dif \theta.
\end{eqnarray*}}
Suppose now that there exists a set $E \subset \T^\infty$
with strictly positive Lebesgue measure such that $\mu^\ast$ is greater than $\varepsilon>0$ there. By the above, this leads to the inequality
\begin{equation*}
	\int_E  f \dif \theta \leq \varepsilon^{-1}  \int f \dif \mu.
\end{equation*}
Since this holds for all positive continuous functions $f$, we easily get  a contradiction. Indeed, shrinking $E$ slightly, if necessary, we may assume that
it is compact and   $\mu(E) = 0$ (by the inner-regularity of Lebesgue measure). From the outside, we may approximate $E$ in $\mu$-measure by an open set $G$ ($\mu$ is outer regular). By Urysohn's lemma, there is a function $f$ which is $0$ outside of $G$ and $1$ on $E$. This yields a contradiction against the previous inequality.

We have shown that $\lim_{k\to\infty} \PO_{r_k}\mu (e^{i\theta})=0$  for
almost every $e^{i\theta}$  when the limit is taken along any sequence $(r_k)$, 
and we still need to improve this to unconstrained convergence along $r\uparrow 1^-.$ For that we now fix the sequence $r_k:=1-k^{-1/3}$  and observe that it is enough to show that there is a constant $C<\infty$ with  
\begin{equation}\label{eq:stayinlimit}
	\frac{ \PO_r(\theta)}{ \PO_{r_k}(\theta)} \leq C\qquad \textrm{for}\quad r \in (r_k, r_{k+1}) \quad  \textrm{and all}\quad k\geq 1.
\end{equation}
Observe first that
\begin{equation*}
 	\frac{  \PO_{r}(\theta)}{  \PO_{r_k}(\theta)} 
	 =
	 \left( \prod_n \frac{1-r^{2n}}{1-r_k^{2n}} \right) \left( \prod_n \frac{1 - 2r_k^n \cos \theta_n + r_k^{2n}}{1 - 2r^n \cos \theta_n + r^{2n}}\right). 
\end{equation*}
Our aim is to apply the general estimate $\Big|\prod_{n=1}^\infty (1+a_n)\Big|\leq\exp\big(\sum_{n=1}^\infty |a_n|\big).$ To this   end, observe first that
$$
\left| \frac{1-r^{2n}}{1-r_k^{2n}}-1\right|= \frac{r^{2n}-r_k^{2n}}{1-r_k^{2n}}\leq \frac{r_{k+1}^{2n}-r_k^{2n}}{1-r_k}
\leq \frac{2n(r_{k+1}-r_k)r_{k+1}^{n-1}}{(1-r_{k+1})^2}.
$$
In a similar vein,
\begin{align*}
\left| \frac{1 - 2r_k^n \cos \theta_n + r_k^{2n}}{1 - 2r^n \cos \theta_n + r^{2n}}-1\right| &= \frac{|2\cos \theta_n (r_k^n - r^n) + (r^{2n} - r_k^{2n})| }{1 - 2r^n \cos \theta_n + r^{2n}} \\
& \leq \frac{4(r_{k+1}^{n}-r_k^{n})r_{k+1}^{n-1}}{(1-r_{k+1}^{n})^2}
\leq \frac{4n(r_{k+1}-r_k)r_{k+1}^{n-1}}{(1-r_{k+1})^2}.
\end{align*}
Hence the required uniform bound in $r$ follows by noting that
\begin{eqnarray*}
\sum_{n=1}^\infty \left(\frac{n(r_{k+1}-r_k)r_{k+1}^{n-1}}{(1-r_{k+1})^2}\right)
=\frac{(r_{k+1}-r_k)}{(1-r_{k+1})^2}\sum_{n=1}^\infty nr_{k+1}^{n-1}
\leq \frac{r_{k+1}-r_k}{(1-r_{k+1})^4}\leq C
\end{eqnarray*}
holds uniformly in $k$ for our choice of   sequence $(r_k).$
\end{proof}

The proofs of the above theorems apply to a more general radial approach:
\textcolor{black}{
\begin{theorem} \label{co:general}  {\rm Theorems \ref{th:fatou} and \ref{th:infiniterudin}} remain valid for the boundary approaches where  the given boundary point $(z_1,z_2,z_3,\ldots )$ is targeted along the curve
$$
(r^{m_1}z_1,r^{m_2}z_2,r^{m_3}z_3,\ldots )
$$
as  $r\to 1^-$ $($and the definitions of the maximal functions etc. are modified accordingly$)$. Here, $(m_j)$ is any sequence of positive integers that satisfies the condition 
$$
A(r):=\sum_{j=1}^\infty r^{m_j}<\infty \quad \textrm{ for all}\quad r<1.
$$
\end{theorem}}
\begin{proof}For Theorem \ref{th:fatou}, this is easy as one observes that the above condition is just what is needed in order to generalise estimate \eqref{eq:mu3}. Actually, it is of interest to note that the same condition is obtained by requiring that for any  point $(z_j)_{j=1}^\infty\in\T^\infty$ one has $(r^{m_j}z_j)\in \ell^2$ for all $r\in(0,1)$, or equivalently that  $(r^{m_j}z_j)\in \ell^1$ for all $r\in (0,1)$. 

In the case of Theorem \ref{th:infiniterudin}, the crucial detail we need to verify is that one still may pick a subsequence $(r_k)$ that increases to $1^-$ in such a way that \eqref{eq:stayinlimit} holds uniformly in $k$. By our previous estimates it is enough to have
$
\displaystyle \sum_{j=1}^\infty m_j\frac{r_{k+1}-r_k}{(1-r_{k+1})^2}r_{k+1}^{m_j-1}\leq 1$, or in other words
\begin{equation}\label{eq:enough}
r_{k+1}-r_k\leq \frac{(1-r_{k+1})^2}{A'(r_{k+1})}\qquad  \textrm{for all}\quad  k\geq 1.
\end{equation}
Note that $A'(r)<\infty $ for $r\in (0,1)$ due to the analyticity of $A.$
We may pick $(r_k)$ inductively  as follows: choose $r_1=1/2$ and select $r_2>r_1$ so that  \eqref{eq:enough} is satisfied for $k=1$. Assume then by induction that  $r_n$, $n\geq 2$ is already chosen so that  \eqref{eq:enough} is satisfied for $k=1,\ldots , n-1$. Let $r'_{n+1}:=(1+r_n)/2$. Denote $b:=(1-r'_{n+1})^2/A'(r'_{n+1})$. If $r'_{n+1}-r_n\leq b$ choose $r_{n+1}=r'_{n+1}$. Otherwise,
set $\ell:=\lfloor (r'_{n+1}-r_n)/b\rfloor$ and set $r_{n+j}=r_n+jb$ for $j=1,\ldots, \ell$. Then, obviously \eqref{eq:enough} holds for
$k=n,\ldots, n+j-1$. Moreover, we have $1-r_{n+j}\leq 3(1-r_{n})$/4 so the inductive construction also  makes sure that $\lim_{k\to\infty}r_k=1$.
\end{proof}

\begin{remark}\label{re:generalizedHelson} Still, another possibility for studying Fatou type theorems in the spirit of the above approach is to use 
the path
$$
(e^{-\lambda_1u}z_1,e^{-\lambda_2u}z_2,e^{-\lambda_3u}z_3,\ldots ):=g_u(z)
$$
where $u\to 0^+.$
This time  one demands that $\sum_{n=1}^\infty e^{-\lambda_nu}<\infty$ for any $u>0,$ and for fixed  $z\in \T^\infty$ one considers 
the harmonic (or analytic) functions $\xi\mapsto  g_\xi(z)$ in the upper half space. In this approach, one has some additional complications stemming from the infinite (Lebesgue) measure of $\R,$ and this is why  above we preferred to work employing an auxiliary parameter $\xi\in\D$ instead.
\end{remark}

\begin{remark}\label{re:rudin}
For harmonic analysis on $\T^d$ in the finite-dimensional case $d<\infty$, the most fundamental  approach to the boundary is the standard radial one. Accordingly, we denote the corresponding boundary value function (at boundary points where the radial boundary value exists) by 
$$
\mu^{**}(\e^{i\theta}):=\lim_{r\to 1^{-}}\mu (re^{i\theta_1},re^{i\theta_2},\ldots , re^{i\theta_d}).
$$
The proofs of both Theorem  \ref{th:fatou} and \ref{th:infiniterudin} obviously work unchanged (and actually simplify considerably) \textcolor{black}{for the   function $\mu^{**}.$} Hence, we obtain:
\begin{corollary} \label{co:rudin} 
 Let $d\mu=fdm+\dif \mu_s$ be a measure on $\T^d$, for $d<\infty$,    where both $f \in L^1(\T^\infty)$ and the singular part $\mu_s$ have Fourier spectrum supported on $\N^\infty_0 \cup (-\N^\infty_0)$. Then  for almost every $e^{it}\in \T^{d}$, it holds that 
$$
\mu^{**}(e^{it})=f(e^{it}).
$$
\end{corollary}
\noindent 
This result is a special case of a the classical theorem by  Marcinkiewicz and Zygmund which holds for all measures on $\T^d$ (see \cite[Theorem 2.3.1]{rudin1969}). Although less general, Corollary \ref{co:rudin} has the advantage of having a nearly trivial proof, while the proof given by Rudin in the reference, which is given explicitly in the case $d=2$, uses a fairly delicate higher-dimensional covering argument.  We also remark that in  the finite dimensional situation,  one knows basically by a standard application of iterated one-dimensional maximal functions that at almost every point the unrestricted radial approach works for boundary functions $f\in  {LLog}^{d-1}(\T^d)$,  see \cite[Chapter 17]{zygmund1968}. 
\end{remark}

We  end this section by observing that Corollary \ref{co:riesz2} holds for a larger class of Hardy spaces, even when $p < 1$. In the classical one-variable $H^p$ theory, this is proved for $p<1$ using inner-outer factorization. As this  tool is not available  in the setting of several variables, a different approach is needed.

Explicitly,  for $p\in (0,1)$, as is usual, the space $L^p(\T^\infty)$ consists of the measurable functions on $\T^\infty$ for which $\int_{\T^\infty} \abs{f}^p \dif \theta$ is finite. It is well known that this is a quasi-Banach space. We define $H^p_{\text{big}}(\T^\infty)$ to be the closure in  $L^p(\T^\infty)$ of those polynomials on $\T^\infty$ for which the Fourier coefficients are supported on $\nu \in \Z^\infty_0$ such that $\sum  \nu_n \geq 0$. As part of the proof of the following corollary, we also verify  that   point evaluation at $0$ is well-defined for these larger Hardy spaces \textcolor{black}{(recall that, as was mentioned in the introduction, that it was shown in  \cite{cole_gamelin1986} that the spaces $H^p(\T^\infty)$ have bounded point evaluations on the set  $\D^\infty \cap \ell^2$).}

\begin{corollary} \label{corr-big}
 Let $p>0$ and assume that ${f}(0) \neq 0$. Then,  $f \in H^p_{\mathrm{big}}(\T^\infty)$ implies $\log \abs{f} \in L^1(\T^\infty)$.
\end{corollary}
\begin{proof}
	Let $f \in H^p_{\text{big}}(\T^\infty).$  By definition, we may pick    a sequence of analytic (in the wider sense described above) polynomials $(P_n)$ so that 
	$\norm{P_n - f}_p \rightarrow 0$ as $n \rightarrow \infty$. For $t\in [0,2\pi)$, we use the notation $\tilde t:=(t,t,\ldots)\in [0,2\pi)^\infty .$ Invoking the measure preserving change of variables $\theta\to \theta +\tilde t$, we obtain
	\begin{equation*}
		\int_{\T^\infty} \abs{P_n(\e^{\im \theta}) - f(\e^{\im \theta})}^p \dif \theta = 
		\int_{\T^\infty} \int_\T \abs{P_n(\e^{\im (\theta+\tilde t)}) - f( \e^{\im (\theta + \tilde t)})}^p  \frac{\dif t}{2\pi} \dif \theta.
	\end{equation*}
	Hence, passing to a subsequence if necessary, a standard argument shows that  for almost every $\theta$, the       one variable polynomial $e^{it}  \longmapsto P_n(\e^{\im (\theta+\tilde t)})$ converges to the function  $\widetilde f_\theta:t \longmapsto f(\e^{\im(\theta+t)})$ in the space $L^p(\T)$
	(observe that the definition of $\widetilde f_\theta$ differs from that of
	$f_\theta$ we applied before).
	Since 
	\begin{equation*}
		P_n(\e^{\im (\theta + \tilde t)}) = \sum_{\text{finite}} a_\nu \e^{\im \nu \cdot \theta} \e^{\im t \sum \nu_n},
	\end{equation*}
	the polynomials $e^{i\theta}  \longmapsto P_n(\e^{\im (\theta+\tilde t)})$ are analytic.	It follows that  $\widetilde f_\theta\in H^p(\T)$ for almost all $\theta$.
If $P\in  H^p_{\mathrm{big}}(\T^\infty)$  is a polynomial, we  observe
that $\widetilde P_\theta (0)=P(0)$ for all $\theta,$ and hence by integrating the one dimensional estimate $|\widetilde P_\theta (0)|^p\leq \int_0^{2\pi}
|\widetilde P_\theta (e^{it})|^p\, dt/2\pi$ over $\theta,$ it follows that
$$
|P(0)|^p\leq C\int_{\T^\infty} |P|^p \dif \theta,
$$
whence  the point evaluation at $0$ is well-defined and bounded on
$ H^p_{\mathrm{big}}(\T^\infty).$ In particular, one has $\widetilde f_\theta(0)=f(0)$ for almost every $\theta\in T^\infty$.

	The same argument that was used to prove Corollary \ref{co:riesz2}, now holds as soon as the relation \eqref{eq:sh} is suitably modified
	using the elementary inequality $\log 1/x \geq - x^p/p$. 
\end{proof}

\section{Examples and open questions} \label{se:examples}

Above, we obtained  positive results for  {\CR specialized} radial approaches in the infinite dimensional situation. As we noted above after Corollary \ref{co:rudin}, in finite dimensions at almost every point the unrestricted radial approach works for functions in $H^p(\T^N)$ with $p>1$. {\CR The content of the following theorem is that this is far from being true in the infinite dimensional case.}

\begin{theorem} \label{th:example} There exists an analytic function $f\in H^\infty(\T^\infty)$, without zeroes,    that fails to have an unrestricted radial limit at almost every boundary point. In fact, $f$ has the following properties:

\smallskip

\noindent{\bf (i)}\quad For almost every point $e^{i\theta}\in\T^{\infty}$, there is a radial approach that is coordinatewise increasing in the sense that for each $n\geq 1$ we have
$r_{n,k}\nearrow 1^-$  as $k\to\infty$  for all $n\geq 1$,  but  under which $f$ fails to converge to the right value $f(e^{i\theta})$. Actually, in this boundary approach one has
$$
\lim_{k\to\infty} |f(r_{1,k}e^{i\theta_1},r_{2,k}e^{i\theta_2},\ldots)| \; =\; 0 \qquad\textrm{for a.e.}\;\; e^{i\theta}\in\T^{\infty}.
$$

\smallskip

\noindent{\bf (ii)}\quad There is a radial approach that is independent of the boundary point, but  under which $f$ fails to converge  to the right value $f(e^{i\theta})$ at almost every boundary point $e^{i\theta}\in\T^{\infty}$.

\smallskip

\end{theorem}

Before explaining how to construct the function described in {\CR the above theorem}, we consider two simpler examples that share many of the same features. In the first, we drop the boundedness, and in the second, we keep boundedness but drop   analyticity.

\begin{example}\label{ex:analytic1} \textbf{(a)}  {\sl  The function
$$
g(z)=\sum_{n=1}^\infty \frac{z_n}{n}
$$
is in $H^p(\T^\infty)$ and 
fails at almost every boundary point to have radial boundary limit in an approach that is independent of the boundary point in the sense of Theorem \ref{th:example} part (ii).}

\textbf{(b)} {\sl The function 
$$
u(z):=\; \prod_{n=1}^\infty\Big(1+i\frac{z_n+\overline{z}_n}{2n}\Big)\;  
$$
is in $L^\infty(\T^\infty)$  and fails at almost every boundary point to have radial boundary limit in the same sense as in} \textbf{(a)}.
\end{example}

\noindent {\CR The most interesting feature of these examples  is that they allow us, in a  simpler setting than in Theorem \ref{th:example}, to explain   how to find a bad radial approach that is independent of the boundary point $\e^{\im \theta} \in \T^\infty$.   As Theorem \ref{th:example} (whose proof we give shortly) covers the phenomena displayed by both examples, we  discuss  only the main details of example \textbf{(a)}.}

{\CR To that end, we first note   that, by the independence of the variables $\e^{\im \theta_n}$,  the series
\begin{equation*}
	\sum_{n=1}^\infty \frac{\Re \e^{\im \theta_n}}{n} = \sum_{n=1}^\infty \frac{  \cos(\theta_n)}{n}
\end{equation*}
is conditionally convergent  almost everywhere on $\T^\infty$ (see, e.g., \cite[Lemma 3.14, p. 46]{kallenberg2002}).  However, since the terms are bounded and 
$\sum_{n=1}^\infty {\mathbf E}\, \frac{  |\cos(\theta_n)|}{n}=\infty$, we deduce from  \cite[Lemma 3.14, p. 46]{kallenberg2002} that at almost every boundary point  the series is \textit{not} absolutely convergent. Hence for every $M \in \N$,   as $N\rightarrow \infty$, we have
\begin{equation*}
	\mathbf{P} \Big(  \sum_{n=M}^N \frac{\abs{\cos(\theta_n)}}{n} \geq 1 \Big) \longrightarrow 1,
\end{equation*}
and  we may use the Borel-Cantelli lemma to inductively choose a sequence $m_1 < m_2 < m_3< \cdots$ so that for almost every $\e^{\im \theta} \in \T^\infty$, and $k\geq k_0$ large enough, we have
\begin{equation*}
	\sum_{m_k+1}^{m_{k+1}} \frac{\abs{\cos (\theta_n)}}{n}  \geq 1.
\end{equation*}}

We now describe the bad approach working for almost every boundary point. For each $k$ we choose a vector $\mathbf{r}_k = (r_{k,1}, r_{2,k}, \ldots)$ as follows. For   $1 \leq k \leq 2^{m_1}$, we {\CR set}
\begin{equation*}
	\hspace{-47.5mm}\mathbf{r}_k = (r_{1,k},  \ldots, r_{m_1,k}, 0, 0, 0, \ldots) 
\end{equation*}
so that the first $m_1$ coordinates run through all $2^{m_1}$ different $m_1$-tuples consisting of $0$ and $1/2$.  

For the next $2^{m_2 - m_1}$ indices $k$, we choose
\begin{equation*}
	\mathbf{r}_k = (\underbrace{1 - m_1^{-1},   \ldots, 1-m_1^{-1}}_{\text{first $m_1$   entries}}, \underbrace{r_{m_{1}+1,k},  \ldots, r_{m_2,k}}_{\text{middle block}},  \textcolor{black}{\underbrace{0, 0, 0, \ldots}_{\text{all zeroes}}})
\end{equation*}
so that the middle block can   run through all tuples consisting of $0$ and $1/2$. 

Similarly, for the next $2^{m_3-m_2}$ coordinates, we choose
\begin{equation*}
	\mathbf{r}_k = (\underbrace{1 - m_2^{-1},   \ldots, 1-m_2^{-1}}_{\text{\textcolor{black}{first $m_2$  entries}}}, \underbrace{r_{m_{2}+1,k},  \ldots, r_{m_3,k}}_{\text{middle block}}, \textcolor{black}{\underbrace{0, 0, 0, \ldots}_{\text{all zeroes}}}), 
\end{equation*}
again, so that the middle block can run through all tuples consisting of $0$ and $1/2$.

If we continue in this way, it is clear that we get a sequence $\mathbf{r}_k$ such that, for fixed $n$, {\CR we have} $r_{n,k} \nearrow 1^-$ as $k \rightarrow \infty$.  Moreover,  given  arbitrarily large $\ell \in \N$,  for a suitable interval  of indices $k$ {\CR it holds that}
\begin{equation*}
	 \Re f(\mathbf{r_k}\e^{\im \theta}) = \Big(1- \frac{1}{m_\ell}\Big) \sum_{n=1}^{m_\ell}  \frac{\cos(\theta_n)}{n} + \sum_{n = m_{\ell} + 1}^{m_{\ell+1}} r_{{\CR m_\ell},k} \frac{\cos(\theta_n)}{n}
\end{equation*}
The first sum remains unchanged as $k$ varies in this interval, but  the second term will   oscillate between values close to $0$, and, in absolute value, larger than $1/2$. Since for almost every fixed boundary point this behaviour takes place infinitely many times, the statement follows.

We remark that the function in part \textbf{(a)} of the  example is very close to being in $H^\infty(\T^\infty)$ in the sense that for some $c>0$ it holds that  $\int_{\T^\infty}\exp(c|f(e^{i\theta})|^2) \dif\theta<\infty$. The reason   this holds is essentially that the Taylor series for this function is lacunary in a strong sense. Namely,   the variables $z_n$ are   independent of each other
The argument for the function in part \textbf{(b)} is   essentially a minor modification of the argument from the first example, although, initially it was inspired by the inductive method to construct Rudin-Shapiro polynomials. \textcolor{black}{We mention that the    uniform bound can be seen by writing}
\begin{equation*}
	\abs{u(z)}  =  \sqrt{\prod_{n=1}^\infty \Big( 1 + \frac{(z+\bar{z})^2}{4n^2}}\Big).
\end{equation*}

We now turn to the proof of Theorem \ref{th:example}.

\begin{proof}[Proof of Theorem \ref{th:example}]
Set $\delta_n:= ((n+2)\log^2(n+2))^{-1}$ for $n\geq 1.$ Pick a smooth, non-negative, and even  function $\psi$ that satisfies $\psi(t)=1$ for  $t\in (-1/4,1/4)$,  
$\psi(t)=0$ for $|t|\geq 1/2$ and $|\psi(t)|\leq 1$ for all $t$. We construct $f$ as the product
$$
f(z)\;{\CR : =}\;\prod_{n=1}^\infty f_n(z_n)\; {\CR : =}\;\prod_{n=1}^\infty\exp\Big(-u_n(z_n)- \im \, \widetilde u_n(z_n)\Big),
$$
where $u_n$ is the positive harmonic function on $\D$ with   boundary values 
$$
u_n(\e^{it}):= \psi (t/\delta_n)\quad\textrm{for}\quad t\in [-\pi,\pi).
$$
One should note that the functions $f_n$ are continuous up to the boundary and real at the origin.

To see that the  product  converges, we may use  the '$m$te Abschnitt'  $A_mf(z) = \prod_{n=1}^m f_n(z_n)$.
From the definition, we obtain that the $A_mf(0)$ converges to a non-zero value since
$$
A_m(0)=\exp\Big( -\sum_{n=1}^m\frac{1}{2\pi}\int_\T u_n(\e^{i\theta_n}) \dif \theta_n \Big) \geq 
\exp\big(-\sum_{n=1}^\infty  \delta_n) >0.
$$
Hence, by  a standard ${\rm weak}^*$ convergence argument, $A_mf$ converges to a non-trivial element  $f \in H^\infty(\T^\infty)$. This can also be seen following an argument of Hilbert which shows that $f$ has bounded point evaluations at all $z \in \D^\infty \cap c_0$ (see \cite{hls1997}). Moreover,  by the Herglotz representation of $-\log f_n(z_n)$ and the fact that $ \frac{1}{2\pi}\int_\T |u_n|\sim \delta_n$, we deduce that  at any point 
$(z_1,z_2,\ldots)\in   \D^\infty \cap c_0$ it holds that 
$$
|f(z)|\geq\exp\big(-C_1\sum_{n=1}^\infty \frac{\delta_n}{1-|z_n|}\big),
$$
so that $f$ is non-vanishing on $c_0\cap \D^\infty$. 

In order to prove (i), we observe that by basic estimates for the Poisson kernel,   the radial maximal function of $u_n$ satisfies  $Mu_n(e^{it})\geq
C_2\min (1,\delta_n|t|^{-1})$ for some constant $C_2>0$ and all $t\in [-\pi, \pi)$. It follows that
 $$
 \int_\T Mu_n  \frac{\dif t}{2\pi} \geq C_2\delta_n\log (1/\delta_n)\geq \frac{C_3}{(n+2)\log(n+2)}.
 $$
In particular, this yields
\begin{equation}\label{eq10001}
\sum_{n=1}^\infty \int_\T Mu_n  \frac{\dif t}{2\pi}\;=\;\infty.
\end{equation}
Since $0\leq Mu_n\leq 1,$  we may use \eqref{eq10001} and  \cite[Lemma 3.14, p. 46]{kallenberg2002}  to infer that
\begin{equation}\label{eq10002}
\sum_{n=1}^\infty Mu_n(\theta_n)\;=\;\infty \qquad \textrm{for almost every}\;\; (\e^{i\theta_n})\in \T^\infty.\nonumber
\end{equation}
In other words, for almost every boundary point $(\e^{i\theta_n})\in \T^\infty$ there are radii $r'_1,r'_2,\ldots <1$ such that
\begin{equation}\label{eq10003}
\sum_{n=1}^\infty u_n(r'_n\theta_n)\;=\;\infty.\nonumber
\end{equation}
We may especially choose an increasing sequence $\nu_\ell$ of indices so that
\begin{equation}
\sum_{n=\nu_{\ell +1}}^{\nu_{\ell+1}}u_n(r'_n\theta_n)\geq 4^\ell\qquad \textrm{for all}\;\; \ell >1.\nonumber
\end{equation}
The desired radial approach for part (i) of the Theorem is   obtained by  choosing for this boundary point
$r_{n,k}:=1$ for $n\leq \nu_k$,  $r_{n,k}:=r'_n$ for $\nu_k<n\leq \nu_{k+1}$, and $r_{n,k}:=0$ for $ \nu_{k+1}< n,$ and finally by slightly perturbing the chosen radii away from 1.

For part (ii) we perform basically the same argument as above, where the role of the maximal function $Mu_n$ is taken by the absolute value of the conjugate function $|\widetilde u_n|$. Namely,  by the definition of the conjugate function  we see that  $|\widetilde u_n|(e^{it})\geq
C_4\min (1,\delta_n|t|^{-1})$ for some constant $C_4>0$, and for  all $t\in [-\pi, \pi)\setminus [-2\delta_n,2\delta_n]$.  As before, it follows that  $
 \int_\T |\widetilde u_n|\geq \delta_n\log (1/\delta_n)\geq \frac{C_5}{(n+2)\log(n+2)} 
 $
 and we obtain the analogue of \eqref{eq10001} for the functions $|\widetilde u_n|$. This together with the independence and the uniform boundedness of the functions $\widetilde u_n$ (recall that $\psi$ is smooth and the Hilbert transform is locally essentially scaling invariant)
 yields,  for almost every boundary point,
\begin{equation}\label{eq10004}
\sum_{n=1}^\infty |\widetilde u_n(i\theta_n)|\;=\;\infty.
\end{equation}
The proof is finished like in Example \ref{ex:analytic1}, described above, and we obtain the desired radial approach where the  fluctuations of $\arg f$ remain large. This would not yield the counterexample in points where $f(e^{i\theta})=0$, but by Corollary \ref{co:riesz2} the measure of such points is zero.
\end{proof}

The following questions appear quite interesting:
\begin{question}
	Is it possible to remove the restriction on the Fourier spectrum in Theorem \ref{th:fatou}?
\end{question}
\begin{question}\label{qu_analytic1}
Does there exist a bounded analytic function $f\in H^\infty (\T^\infty)$ such that almost surely the radial convergence fails even if  the approach is limited by assuming decreasing radii in $n$, i.e. $r_{n,k}\geq r_{n+1,k}$ for all $n,k$? 
\end{question}
\begin{question}
What is the answer to the above question under the  added  condition that the radial approach does not depend on the point on the boundary?
\end{question}
\begin{question}
What can one say about non-tangential approach for functions in $H^\infty(\T^\infty)$,  or more generally,  in $H^p(\T^\infty)$?  
\end{question}
\begin{question}
	For which radial approaches is the brothers Riesz   theorem on the uniqueness of boundary values true? That is,  when is it true for all $f \in H^\infty(\T^\infty)$ that having   vanishing radial boundary values on a set of positive Lebesgue measure implies   $f \equiv 0$? \textcolor{black}{Note that this is not true for all radial approaches as is demonstrated by Theorem \ref{th:example}.}
\end{question}

\section*{Acknowledgements}

The authors would like to thank Micha\l{} Wojciechowski for pointing out an error in a previous version of the paper.

\bibliographystyle{amsplain}
\def\cprime{$'$} \def\cprime{$'$} \def\cprime{$'$} \def\cprime{$'$}
\providecommand{\bysame}{\leavevmode\hbox to3em{\hrulefill}\thinspace}
\providecommand{\MR}{\relax\ifhmode\unskip\space\fi MR }
\providecommand{\MRhref}[2]{%
  \href{http://www.ams.org/mathscinet-getitem?mr=#1}{#2}
}
\providecommand{\href}[2]{#2}

\end{document}